\documentclass{amsart}
\usepackage{amssymb}
\usepackage{amscd}
\usepackage{multicol}
\usepackage{tikz}
\usepackage{graphics}
\usepackage{shuffle}
\usepackage[OT2,T1]{fontenc}
\usepackage[all]{xy}
\usepackage{extarrows}
\usepackage{comment}
\DeclareSymbolFont{cyrletters}{OT2}{wncyr}{m}{n}
\DeclareMathSymbol{\Sha}{\mathalpha}{cyrletters}{"58}

\title[On shuffle-type functional relations of desingularized MZFs]
{On shuffle-type functional relations of desingularized multiple zeta-functions\\
}
\author{Nao Komiyama}

\address{Graduate School of Mathematics, Nagoya University, 
Furo-cho, Chikusa-ku, Nagoya 464-8602 Japan }
\email{m15027u@math.nagoya-u.ac.jp}
\date{February 18, 2020}

\newtheorem{thm}{Theorem}[section]
\newtheorem{lmm}[thm]{Lemma}

\newtheorem{prp}[thm]{Proposition}  

\theoremstyle{remark}
        

\numberwithin{equation}{section}
\theoremstyle{definition}
\newtheorem{dfn}[thm]{Definition}
\newtheorem{rem}[thm]{Remark}

\newcommand{\Bold}[1]{\mbox{\boldmath #1}}

\begin{document}
\bibliographystyle{amsalpha+}
\maketitle

\begin{abstract}      
We treat desingularized multiple zeta-functions introduced by Furusho, Komori, Matsumoto and Tsumura. In this paper, we prove functional relations, which are shuffle type product formulae, between desingularized multiple zeta-functions and desingularized values.
\end{abstract}

\setcounter{section}{-1}
\section{Introduction}


We start with the {\it multiple zeta-function} (MZF for short), which is defined by
\begin{equation*}
	\zeta(s_1, \dots, s_r) := \sum_{0<m_1<\cdots<m_r}\frac{1}{m_1^{s_1}\cdots m_r^{s_r}}.
\end{equation*}
It converges absolutely in the region
\begin{equation*}
	\{(s_1,\dots,s_r)\in\mathbb{C}^r\ |\ \frak{R}(s_{r-k+1}+\cdots+s_r)>k\ (1\leq k\leq r)\}.
\end{equation*}
The special values of this function at $(s_1,\dots,s_r)=(k_1,\dots,k_r)$ for $k_1,\dots,k_{r-1}\geq1$ and $k_r\geq2$ are called {\it multiple zeta values} (MZVs for short). They were studied by  Euler (\cite{Euler}), Ecalle (\cite{Eca}), Hoffman (\cite{Hof}), Zagier (\cite{Zag}), etc.

 In the early 2000s, Zhao (\cite{Zhao}) and Akiyama, Egami and Tanigawa (\cite{AET}) independently showed that MZF can be meromorphically continued to $\mathbb{C}^r$. Especially, in \cite{AET}, the set of all singularities of the function $\zeta(s_1,\dots,s_r)$ is explicitly given as
	\begin{align*}
		&s_r=1,\nonumber\\
		&s_{r-1}+s_r=2,1,0,-2,-4,\dots,\\
		&s_{r-k+1}+\cdots+s_r=k-n\quad (3\leq k\leq r,\ n\in\mathbb{N}_0).\nonumber
	\end{align*}
Because almost all of non-positive integer points are located in the above singularities, we can not determine the special value of this function at these points.
Only the special values $\zeta(-k)$ (at $k\in\mathbb{N}_0$) and $\zeta(-k_1,-k_2)$ (at $k_1,k_2\in\mathbb{N}_0$ with $k_1+k_2$ odd) are well-defined.
Giving a nice definition of ``$\zeta(-k_1,\dots,-k_r)$'' for $k_1,\dots,k_r\in\mathbb{N}_0$ is one of our most fundamental problems.

In \cite{FKMT1}, Furusho, Komori, Matsumoto and Tsumura introduced the {\it desingularized MZF} $\zeta_r^{\rm des}(s_1,\dots,s_r)$ which is entire on the whole space $\mathbb{C}^r$. They proved that the functions are represented by finite linear combinations of shifted MZFs (cf. Proposition \ref{prp:1.1}). They also showed explicit formulae of special values of $\zeta_r^{\rm des}(s_1,\dots,s_r)$ at non-positive integers, which is called {\it desingularized values}, in terms of the Bernoulli numbers (see Proposition \ref{prp:1.1.1}).

In the previous paper \cite{Komi2}, 
the author showed the following sfuffle-type product formulae of desingularized values at non-positive integers which is based on a equivalence in \cite[Theorem 3.5]{Komi} between renormalized values by Ebrahimi-Fard, Manchon, Singer \cite{EMS1} and desingularized values in \cite{FKMT1}:\\

\noindent
{\bf Theorem \ref{prop:2.1}} ({\rm \cite[Theorem 3.3]{Komi2}})
{\it For $p,q\in\mathbb{N}$ and $k_1,\dots,k_p,l_1,\dots,l_q\in\mathbb{N}_0$, we have}
	\begin{align*}
		&&\zeta_p^{\rm des}(-k_1,\dots,-k_p)\zeta_q^{\rm des}(-l_1,\dots,-l_q)\hspace{6.6cm} \\
		&&= \sum_{\substack{
			i_b + j_b=l_b \\
			i_b,j_b\geq0 \\
			1\leq b \leq q}}
		\prod_{a=1}^q(-1)^{i_a}\binom{l_a}{i_a} \zeta_{p+q}^{\rm des}(-k_1,\dots,-k_{p-1},-k_p-i_1- \cdots -i_q,-j_1,\dots,-j_q). \nonumber
	\end{align*}
He also proved the following shuffle-type functional relations between desingularized MZFs and desingularized values:\\

\noindent
{\bf Proposition \ref{prop:2.2}} ({\rm \cite[Proposition 4.9]{Komi2}})
{\it For $s_1,\dots,s_{r-1}\in\mathbb{C}$ and $l\in\mathbb{N}_0$, we have}
	\begin{equation*}
		\zeta_{r-1}^{\rm des}(s_1,\dots,s_{r-1})\zeta_1^{\rm des}(-l)=\sum_{i+j=l}(-1)^i\binom{l}{i}\zeta_r^{\rm des}(s_1,\dots,s_{r-2},s_{r-1}-i,-j).
	\end{equation*}
In this paper, we show the following theorem which generalizes the above two formulae. 
\\

\noindent
{\bf Theorem \ref{cor:4.1}}
{\it For $s_1,\dots,s_p\in\mathbb{C}$ and $l_1,\dots,l_q\in\mathbb{N}_0$, we have}
	\begin{align*}
		&\zeta_p^{\rm des}(s_1,\dots,s_p)\zeta_q^{\rm des}(-l_1,\dots,-l_q)\hspace{6.6cm} \\
		&= \sum_{\substack{
			i_b + j_b=l_b \\
			i_b,j_b\geq0 \\
			1\leq b \leq q}}
		\prod_{a=1}^q(-1)^{i_a}\binom{l_a}{i_a} \zeta_{p+q}^{\rm des}(s_1,\dots,s_{p-1},s_p-i_1- \cdots -i_q,-j_1,\dots,-j_q). 
	\end{align*}

The plan of our paper goes as follows. In \S1, we will recall the definition of the desingularized MZFs and some properties of these functions. In \S2, we will show the formula (\ref{eqn:4.6}) in Theorem \ref{cor:4.1}, which is reduced to Proposition \ref{prop:2.1} and Proposition \ref{prop:2.2}.

\section{Desingularization of multiple zeta-functions}
In this section, we review desingularized MZFs desingularized values introduced by Furusho, Komori, Matsumoto and Tsumura in \cite{FKMT1}. 
We recall the definition of the desingularized MZF, and explain  some remarkable properties of their functions. 

We start with the generating function\footnote{It is denoted by $\tilde{\mathfrak{H}}_n\left((t_j);(1);c\right)$ in \cite{FKMT1}.} $\tilde{\mathfrak{H}}_r\left(t_1,\dots,t_r;c\right) \in \mathbb{C}[[t_1,\dots,t_r]]$ (cf. \cite[Definition 1.9]{FKMT1}):
\begin{align*}
	\tilde{\mathfrak{H}}_r\left(t_1,\dots,t_r;c\right)&:=\prod_{j=1}^r\left(\frac{1}{\exp{\left(\sum_{k=j}^r t_k\right)}-1}-\frac{c}{\exp{\left(c\sum_{k=j}^r t_k\right)}-1}\right)\\
	&=\prod_{j=1}^r\left(\sum_{m=1}^{\infty}(1-c^m)B_m\frac{\left(\sum_{k=j}^r t_k\right)^{m-1}}{m!}\right)
\end{align*}
for $c\in\mathbb{R}$. Here $B_m\ (m\geq0)$ is the Bernoulli number which is defined by
\begin{equation}\label{eqn:1.1.4}
\displaystyle\frac{x}{e^x-1}:=\sum_{m\geq0}\frac{B_m}{m!}x^m.
\end{equation}
We note that $B_0=1$, $B_1=-\frac{1}{2}$, $B_2=\frac{1}{6}$.
\begin{dfn}[{\rm \cite[Definition 3.1]{FKMT1}}]
	For non-integral complex numbers $s_1,\dots,s_r$, {\it desingularized MZF} $\zeta_r^{\rm des}(s_1,\dots,s_r)$ is defined by
	\begin{align}
		\label{eqn:1.1.2}&\zeta_r^{\rm des}(s_1,\dots,s_r) \\
		&:=\lim_{\substack{c\rightarrow1\\c\in\mathbb{R}\setminus\{1\}}}\frac{1}{(1-c)^r}\prod_{k=1}^r\frac{1}{(e^{2\pi is_k}-1)\Gamma(s_k)}\int_{\mathcal{C}^r}\tilde{\mathfrak{H}}_r\left(t_1,\dots,t_r;c\right)\prod_{k=1}^r t_k^{s_k-1}d t_k. \nonumber
	\end{align}
	Here $\mathcal{C}$ is the path consisting of the positive real axis (top side), a circle around the origin of radius $\varepsilon$ (sufficiently small), and the positive real axis (bottom side).
\end{dfn}
One of the remarkable properties of desingularized MZF is that it is an entire function, i.e., the equation (\ref{eqn:1.1.2}) is well-defined as an analytic function by the following proposition.
\begin{prp}[{\rm \cite[Theorem 3.4]{FKMT1}}]\label{prp:1.2}
	The function $\zeta_r^{\rm des} (s_1,\dots,s_r)$ can be analytically continued to $\mathbb{C}^r$ as an entire function in $(s_1,\dots,s_r)\in \mathbb{C}^r$ by the following integral expression:
	\begin{align*}
		\label{eqn:1.1.2}\zeta_r^{\rm des}&(s_1,\dots,s_r) 
		=\prod_{k=1}^r\frac{1}{(e^{2\pi is_k}-1)\Gamma(s_k)}\\
		&\cdot\int_{\mathcal{C}^n}\prod_{j=1}^r\lim_{\substack{c\rightarrow1\\c\in\mathbb{R}\setminus\{1\}}}\frac{1}{1-c}\left(\frac{1}{\exp{\left(\sum_{k=j}^r t_k\right)}-1}-\frac{c}{\exp{\left(c\sum_{k=j}^r t_k\right)}-1}\right)\prod_{k=1}^r t_k^{s_k-1}d t_k.
	\end{align*}
\end{prp}
For indeterminates $u_j$ and $v_j\ (1\leq j\leq r)$, we set
\begin{equation}\label{eqn:1.1.3}
	\mathcal{G}_r(u_1,\dots,u_r; v_1,\dots,v_r):=\prod_{j=1}^r\left\{1-(u_jv_j+\cdots+u_r v_r)(v_j^{-1}-v_{j-1}^{-1})\right\}
\end{equation}
with the convention $v_0^{-1}:=0$, and we define the set of integers $\{a^r_{\Bold{\footnotesize$l$},\Bold{\footnotesize$m$}}\}$ by
\begin{equation}\label{eqn:1.1.4}
	\mathcal{G}_r(u_1,\dots,u_r; v_1,\dots,v_r)=\sum_{\substack{\mbox{\boldmath {\footnotesize$l$}}=(l_j)\in\mathbb{N}_0^r\\ \mbox{\boldmath {\footnotesize$m$}}=(m_j)\in\mathbb{Z}^r \\ |\mbox{\boldmath {\footnotesize$m$}}|=0}}a^r_{\mbox{\boldmath {\footnotesize$l$}},\mbox{\boldmath {\footnotesize$m$}}}\prod_{j=1}^ru_j^{l_j}v_j^{m_j}.
\end{equation}
Here, $|\mbox{\boldmath {\footnotesize$m$}}|:=m_1+\cdots+ m_r$.\\
Another remarkable properties of desingularized MZF is that the function is given by a finite ``linear'' combination of shifted MZFs, i.e.,
\begin{prp}[{\rm \cite[Theorem 3.8]{FKMT1}}]\label{prp:1.1}
	For $s_1,\dots,s_r \in \mathbb{C}$, we have the following equality between meromorphic functions of the complex variables $(s_1,\ldots,s_r)$:
	\begin{equation}\label{eqn:1.1.5}
		\zeta_r^{\rm des}(s_1,\dots,s_r)=\sum_{\substack{\mbox{\boldmath {\footnotesize$l$}}=(l_j)\in\mathbb{N}_0^r\\ \mbox{\boldmath {\footnotesize$m$}}=(m_j)\in\mathbb{Z}^r \\ |\mbox{\boldmath {\footnotesize$m$}}|=0}}a^r_{\mbox{\boldmath {\footnotesize$l$}},\mbox{\boldmath {\footnotesize$m$}}}\left(\prod_{j=1}^r(s_j)_{l_j}\right)\zeta(s_1+m_1,\dots,s_r+m_r).
	\end{equation}
	Here, $(s)_{k}$ is the {\it Pochhammer symbol}, that is, for $k\in\mathbb{N}$ and $s\in\mathbb{C}$ $(s)_{0}:=1$ and $(s)_k:=s(s+1)\cdots(s+k-1)$.
\end{prp}

\begin{dfn}\label{def:1.2.1}
	For $k_1,\dots,k_r \in \mathbb{N}_0$, {\it desingularized value} $\zeta_r^{\rm des}(-k_1,\dots,-k_r)\in\mathbb{C}$ is defined to be the special value of desingularized MZF $\zeta_r^{\rm des}(s_1,\dots,s_r)$ at $(s_1,\dots,s_r)=(-k_1,\dots,-k_r)$.
\end{dfn}
We consider the following generating function $Z_{\scalebox{0.5}{\rm FKMT}}(t_1,\dots,t_r)$ of $\zeta_r^{\rm des}(-k_1,\dots,-k_r)$ which is defined by
\begin{equation*}
	Z_{\scalebox{0.5}{\rm FKMT}}(t_1,\dots,t_r) := \sum_{k_1,\dots,k_r=0}^{\infty}\frac{(-t_1)^{k_1}\cdots(-t_r)^{k_r}}{k_1!\cdots k_r!}\zeta_r^{\rm des}(-k_1,\dots,-k_r).
\end{equation*}
This is explicitly calculated as follows.
\begin{prp}[{\rm \cite[Theorem 3.7]{FKMT1}}]\label{prp:1.1.1}
	We have
	\begin{equation*}
		Z_{\scalebox{0.5}{\rm FKMT}}(t_1,\dots,t_r) = \prod_{i=1}^r\frac{(1-t_i-\cdots-t_r)e^{t_i+\cdots+t_r}-1}{(e^{t_i+\cdots+t_r}-1)^2}.
	\end{equation*}
	In terms of $\zeta_r^{\rm des}(-k_1,\dots,-k_r)$ for $k_1,\dots,k_r\in\mathbb{N}_0$, the above equation is reformulated to
	\begin{equation*}
		\zeta_r^{\rm des}(-k_1,\dots,-k_r)=(-1)^{k_1+\cdots+k_r}\sum_{\substack{\nu_{1i}+\cdots+\nu_{ii}=k_i\\1\leq i\leq r}}\prod_{i=1}^r\frac{k_i!}{\prod_{j=i}^r\nu_{ij}!}B_{\nu_{ii}+\cdots+\nu_{ir}+1}.
	\end{equation*}
\end{prp}

We have the following shuffle-type product formulae of desingularized values at non-positive integer points.
\begin{thm}[{\rm \cite[Theorem 3.3]{Komi2}}]\label{prop:2.1}
For $p,q\in\mathbb{N}$ and $k_1,\dots,k_p,l_1,\dots,l_q\in\mathbb{N}_0$, we have
	\begin{align}\label{eqn:1.2}
		&&\zeta_p^{\rm des}(-k_1,\dots,-k_p)\zeta_q^{\rm des}(-l_1,\dots,-l_q)\hspace{6.6cm} \\
		&&= \sum_{\substack{
			i_b + j_b=l_b \\
			i_b,j_b\geq0 \\
			1\leq b \leq q}}
		\prod_{a=1}^q(-1)^{i_a}\binom{l_a}{i_a} \zeta_{p+q}^{\rm des}(-k_1,\dots,-k_{p-1},-k_p-i_1- \cdots -i_q,-j_1,\dots,-j_q). \nonumber
	\end{align}
\end{thm}

\begin{rem}
The desingularized value $\zeta_r^{\rm des}(-k_1,\dots,-k_r)$ satisfies the same shuffle-type product formula to $\zeta_{\scalebox{0.5}{\rm EMS}}(-k_1,\dots,-k_r)$ introduced in \cite{EMS1}.
\end{rem}

In \cite{Komi2}, the author generalize the above proposition to the following.
\begin{prp}[{\rm \cite[Proposition 4.9]{Komi2}}]\label{prop:2.2}
	For $s_1,\dots,s_{r-1}\in\mathbb{C}$ and $l\in\mathbb{N}_0$, we have
	\begin{equation*}\label{eqn:4.1}
		\zeta_{r-1}^{\rm des}(s_1,\dots,s_{r-1})\zeta_1^{\rm des}(-l)=\sum_{i+j=l}(-1)^i\binom{l}{i}\zeta_r^{\rm des}(s_1,\dots,s_{r-2},s_{r-1}-i,-j).
	\end{equation*}
\end{prp}
In the next section, we will show the generalization this proposition.

\section{Functional relation of desingularized MZF}
We prove shuffle-type product formulae between $\zeta_p^{\rm des}(s_1,\dots,s_p)$ and
 $\zeta_q^{\rm des}(-l_1,\dots,-l_q)$ in Theorem \ref{cor:4.1}. We assume $r\in\mathbb{N}_{\geq2}$ in this section. In \cite{FKMT1}, the {\it multiple zeta-function of the generalized Euler-Zagier type} is defined by
\begin{align*}
	\zeta_r(s_1,\dots,s_r;\gamma_1,\dots,\gamma_r):=\sum_{m_1,\dots,m_r\geq1}\prod_{k=1}^r\left(\gamma_{1}m_{1}+\cdots+\gamma_km_k\right)^{-s_k},
\end{align*}
for $\gamma_1,\dots,\gamma_r\in\mathbb{C}$ with the condition $\Re(\gamma_j)>0\quad (1\leq j\leq r)$.
This series absolutely converges in the region
\begin{equation}\label{eqn:4.1}
	\{(s_1,\dots,s_r)\in\mathbb{C}^r\ |\ \Re(s_{r-k+1}+\cdots+s_r)>k\ (1\leq k\leq r)\}.
\end{equation}
In \cite{Matsumoto}, it is proved that this function $\zeta_r(s_1,\dots,s_r;\gamma_1,\dots,\gamma_r)$ can be meromorphically continued to $\mathbb{C}^r$.
For simplicity, we sometimes denote it by $\zeta_r((s_j);(\gamma_j))$. 
\begin{lmm}\label{lmm:4.1}
	For $s_1,\dots,s_r\in\mathbb{C}$, we have
	{\small
	\begin{align}\label{eqn:4.8}
		\zeta_r^{\rm des}(s_1,\dots,s_r)=\lim_{\substack{c\rightarrow1 \\ c\in\mathbb{R}\setminus\{1\}}}\frac{1}{(1-c)^r}\sum_{\delta_1,\dots,\delta_r\in\{0,1\}}(-c)^{\delta_1+\cdots+\delta_r}\zeta_r(s_1,\dots,s_r;c^{\delta_1},\dots,c^{\delta_r}).
	\end{align}}
\end{lmm}
\begin{proof}
	Let $c>0$ such that $|c-1|$ is sufficiently small. We assume $(s_1,\dots,s_r)\in\mathbb{C}^r$ satisfies
	\begin{equation*}
		\Re(s_{r-k+1}+\cdots+s_r)>k\quad(1\leq k\leq r).
	\end{equation*}
	Then, we have
	{\small
	\begin{align}\label{eqn:4.9}
		&\lim_{\substack{c\rightarrow1 \\ c\in\mathbb{R}\setminus\{1\}}}\frac{1}{(1-c)^r}\sum_{\delta_1,\dots,\delta_r\in\{0,1\}}(-c)^{\delta_1+\cdots+\delta_r}\zeta_r(s_1,\dots,s_r;c^{\delta_1},\dots,c^{\delta_r}) \\
		&=\lim_{\substack{c\rightarrow1 \\ c\in\mathbb{R}\setminus\{1\}}}\frac{1}{(1-c)^r}\sum_{\delta_1,\dots,\delta_r\in\{0,1\}}(-c)^{\delta_1+\cdots+\delta_r}\sum_{m_1,\dots,m_r\geq1}\prod_{k=1}^r\left(c^{\delta_1}m_1+\cdots+c^{\delta_k}m_k\right)^{-s_k}.  \nonumber
	\end{align}}
	Because we have 
	\begin{equation*}
		m^{-s}=\frac{1}{\Gamma(s)}\int_0^{\infty}e^{-tm}t^{s-1}dt
	\end{equation*}
	by using the Mellin transformation, we get
	\begin{align*}
		\lim_{\substack{c\rightarrow1 \\ c\in\mathbb{R}\setminus\{1\}}}&\frac{1}{(1-c)^r}\sum_{\delta_1,\dots,\delta_r\in\{0,1\}}(-c)^{\delta_1+\cdots+\delta_r}\zeta_r(s_1,\dots,s_r;c^{\delta_1},\dots,c^{\delta_r}) \\
		=&\lim_{\substack{c\rightarrow1 \\ c\in\mathbb{R}\setminus\{1\}}}\frac{1}{(1-c)^r}\sum_{\delta_1,\dots,\delta_r\in\{0,1\}}(-c)^{\delta_1+\cdots+\delta_r} \\
		&\hspace{2.5cm}\cdot \sum_{m_1,\dots,m_r\geq1}
		 \prod_{k=1}^r\left\{\frac{1}{\Gamma(s_k)}\int_0^{\infty} e^{-t_k\sum_{j=1}^kc^{\delta_j}m_j} t_k^{s_k-1}d t_k\right\} \\
		=&\lim_{\substack{c\rightarrow1 \\ c\in\mathbb{R}\setminus\{1\}}}\frac{1}{(1-c)^r}\sum_{\delta_1,\dots,\delta_r\in\{0,1\}}(-c)^{\delta_1+\cdots+\delta_r}\sum_{m_1,\dots,m_r\geq1} \\
		&\hspace{2.5cm}\cdot \prod_{k=1}^r\frac{1}{\Gamma(s_k)}\int_{(0,\infty)^r} \prod_{n=1}^re^{-t_n\sum_{j=1}^nc^{\delta_j}m_j} \prod_{l=1}^r t_l^{s_l-1}d t_l.
		\intertext{By using $\displaystyle\prod_{n=1}^re^{-t_n\sum_{j=1}^nc^{\delta_j}m_j}=\prod_{j=1}^re^{-m_jc^{\delta_j}\sum_{n=j}^rt_n}$, we have}
		=&\lim_{\substack{c\rightarrow1 \\ c\in\mathbb{R}\setminus\{1\}}}\frac{1}{(1-c)^r}\sum_{\delta_1,\dots,\delta_r\in\{0,1\}}(-c)^{\delta_1+\cdots+\delta_r}\sum_{m_1,\dots,m_r\geq1} \\
		&\hspace{2.5cm}\cdot \prod_{k=1}^r\frac{1}{\Gamma(s_k)}\int_{(0,\infty)^r} \prod_{j=1}^re^{-m_jc^{\delta_j}\sum_{n=j}^rt_n} \prod_{l=1}^r t_l^{s_l-1}d t_l.
	\end{align*}
	Because $\zeta((s_j);(c^{\delta_j}))$ absolutely converges, the integral $\int_{(0,\infty)^r}$ and the sum $\sum_{m_1,\dots,m_r\geq1}$ are commutative. So we have
	\begin{align*}
		\lim_{\substack{c\rightarrow1 \\ c\in\mathbb{R}\setminus\{1\}}}&\frac{1}{(1-c)^r}\sum_{\delta_1,\dots,\delta_r\in\{0,1\}}(-c)^{\delta_1+\cdots+\delta_r}\zeta_r(s_1,\dots,s_r;c^{\delta_1},\dots,c^{\delta_r}) \\
		&=\lim_{\substack{c\rightarrow1 \\ c\in\mathbb{R}\setminus\{1\}}}\frac{1}{(1-c)^r}\prod_{k=1}^r\frac{1}{\Gamma(s_k)} \\
		&\hspace{.5cm}\cdot \int_{(0,\infty)^r} \sum_{\delta_1,\dots,\delta_r\in\{0,1\}}(-c)^{\delta_1+\cdots+\delta_r}\sum_{m_1,\dots,m_r\geq1}\prod_{j=1}^re^{-m_jc^{\delta_j}\sum_{n=j}^rt_n} \prod_{l=1}^r t_l^{s_l-1}d t_l \\
		&=\lim_{\substack{c\rightarrow1 \\ c\in\mathbb{R}\setminus\{1\}}}\frac{1}{(1-c)^r}\prod_{k=1}^r\frac{1}{\Gamma(s_k)} \\
		&\hspace{.5cm}\cdot \int_{(0,\infty)^r}\prod_{j=1}^r\left\{ \sum_{\delta_j\in\{0,1\}}(-c)^{\delta_j}\sum_{m_j\geq1}e^{-m_jc^{\delta_j}\sum_{n=j}^rt_n} \right\}\prod_{l=1}^r t_l^{s_l-1}d t_l. \\
		\intertext{By using the defintion of $\tilde{\mathfrak{H}}_r$ and the following formula
			\begin{equation}\label{eqn:2}
				\frac{1}{e^y-1}-\frac{c}{e^{cy}-1}=\sum_{m\geq1}e^{-my}-c\sum_{m\geq1}e^{-mcy}=\sum_{\delta\in\{0,1\}}(-c)^{\delta}\sum_{m\geq1}e^{-mc^{\delta}y},
			\end{equation}
		for $y>0$, we get}
		&\lim_{\substack{c\rightarrow1 \\ c\in\mathbb{R}\setminus\{1\}}}\frac{1}{(1-c)^r}\sum_{\delta_1,\dots,\delta_r\in\{0,1\}}(-c)^{\delta_1+\cdots+\delta_r}\zeta_r(s_1,\dots,s_r;c^{\delta_1},\dots,c^{\delta_r}) \\
		&=\lim_{\substack{c\rightarrow1 \\ c\in\mathbb{R}\setminus\{1\}}}\frac{1}{(1-c)^r}\prod_{k=1}^r\frac{1}{\Gamma(s_k)}\int_{(0,\infty)^r}\tilde{\mathfrak{H}}_r\left(t_1,\dots,t_r;c\right)\prod_{l=1}^r t_l^{s_l-1}d t_l \\
		&=\zeta_r^{\rm des}(s_1,\dots,s_r).
	\end{align*}
	Therefore, we get the claim for $(s_1,\dots,s_r)\in\mathbb{C}^r$ with
	\begin{equation*}
		\Re(s_{r-k+1}+\cdots+s_r)>k\quad (1\leq k\leq r).
	\end{equation*}
	Because $\zeta_r^{\rm des}(s_1,\dots,s_r)$ and $\zeta_r((s_j);(\gamma_j))$ are meromorphic on $\mathbb{C}^r$ and the limit of meromorphic functions is also meromorphic, the equation (\ref{eqn:4.8}) holds for $(s_1,\dots,s_r)\in\mathbb{C}^r$.
\end{proof}
\begin{lmm}\label{lmm:4.2}
	Let $\gamma_1,\dots,\gamma_r>0$, $s_1,\dots,s_r\in\mathbb{C}$ with $\Re(s_j)>1$ $(1\leq j \leq r)$. Put $1\leq t\leq r-1$ and take $a_{t+1},\dots,a_r\in\mathbb{R}$ with $-\Re(s_k)<a_k<0$ $(t+1\leq k \leq r)$. Then, we have 
\begin{align}\label{eqn:4.3}
	\zeta_r((s_j);(\gamma_j)) =& \left(\frac{1}{2\pi i}\right)^{r-t}\int_{(a_{t+1})\times\cdots\times(a_r)}\prod_{k=t+1}^r\frac{\Gamma(s_k+z_k)\Gamma(-z_k)}{\Gamma(s_k)} \\
	&\hspace{0.5cm}\cdot \zeta_t\left(s_1,\dots,s_{t-1},s_t+\sum_{j=t+1}^r(s_j+z_j);\gamma_1,\dots,\gamma_t\right) \nonumber\\
	&\hspace{1cm}\cdot\zeta_{r-t}(-z_{t+1},\dots,-z_{r};\gamma_{t+1},\dots,\gamma_r)\prod_{l=t+1}^rdz_l. \nonumber
\end{align}
	Here, the symbol $(a_k)$ is the path of integration on the vertical line $\Re{(z_k)}=a_k$ from $a_k-i\infty$ to $a_k+i\infty$, for $t+1\leq k\leq r$.
\end{lmm}
\begin{proof}
Consider Mellin-Barnes integral formula
\begin{equation*}
(1+\lambda)^{-s}=\frac{1}{2\pi i}\int_{(a)}\frac{\Gamma(s+z)\Gamma(-z)}{\Gamma(s)}\lambda^zdz,
\end{equation*}
where $\lambda,s\in\mathbb{C}$, $\lambda\neq0$, $|\arg \lambda|<\pi$, $\Re(s)>0$, $-\Re(s)<a<0$.

For $m_1,\dots,m_r\geq1$, by putting $\lambda=\frac{\gamma_{t+1}m_{t+1}+\cdots+\gamma_jm_j}{\gamma_1m_1+\cdots+\gamma_tm_t}$ and $s=s_j$ and $a=a_j$ for $j=t+1,\dots,r$ ($1\leq t\leq r-1$), we have
\begin{equation*}
	\left(\frac{\gamma_1m_{1}+\cdots+\gamma_jm_j}{\gamma_1m_1+\cdots+\gamma_tm_t}\right)^{-s_j}=\frac{1}{2\pi i}\int_{(a_j)}\frac{\Gamma(s_j+z_j)\Gamma(-z_j)}{\Gamma(s_j)}\left(\frac{\gamma_{t+1}m_{t+1}+\cdots+\gamma_jm_j}{\gamma_1m_1+\cdots+\gamma_tm_t}\right)^{z_j}dz_j.
\end{equation*}
So we get
\begin{align*}
	&\left(\gamma_1m_{1}+\cdots+\gamma_jm_j\right)^{-s_j} \\
	&=\frac{1}{2\pi i}\int_{(a_j)}\frac{\Gamma(s_j+z_j)\Gamma(-z_j)}{\Gamma(s_j)}(\gamma_1m_1+\cdots+\gamma_tm_t)^{-s_j-z_j}\left(\gamma_{t+1}m_{t+1}+\cdots+\gamma_jm_j\right)^{z_j}dz_j.
\end{align*}
Taking product over $j=t+1,\dots,r$ and taking summation over $m_{t+1},\dots,m_r\geq1$, we have
\begin{align}\label{eqn:4.2.1}
	&\sum_{m_{t+1},\dots,m_r\geq1}\prod_{j=t+1}^r\left(\gamma_1m_{1}+\cdots+\gamma_jm_j\right)^{-s_j} \\
	&=\left(\frac{1}{2\pi i}\right)^{r-t}\sum_{m_{t+1},\dots,m_r\geq1}\int_{(a_{t+1})\times\cdots\times(a_r)}\prod_{k=t+1}^r\frac{\Gamma(s_k+z_k)\Gamma(-z_k)}{\Gamma(s_k)} \nonumber\\
	&\hspace{0cm}\cdot (\gamma_1m_1+\cdots+\gamma_tm_t)^{-\sum_{j=t+1}^r(s_j+z_j)}\prod_{j=t+1}^r\left(\gamma_{t+1}m_
	{t+1}+\cdots+\gamma_jm_j\right)^{z_j}\prod_{l=t+1}^rdz_l. \nonumber
\end{align}

By multiplying $\prod_{j=1}^t(\gamma_1m_{1}+\cdots+\gamma_jm_j)^{-s_j}$ to the equation (\ref{eqn:4.2.1}) and taking summation over $m_1,\dots,m_t\geq1$, we see that the left hand side becomes $\zeta_r((s_j);(\gamma_j))$. The series $\zeta_r((s_j);(\gamma_j))$ absolutely converges in the region (\ref{eqn:4.1}) and we have $\Re(s_j)>1$\ ($1\leq j\leq r$), so we get the equation (\ref{eqn:4.3}).
\end{proof}
We set $(-z_1,\dots,-z_t):=\left(s_1,\dots,s_{t-1},s_t+\sum_{j=t+1}^r(s_j+z_j)\right)$.
\begin{lmm}\label{lmm:4.3}
	Let $c\in\mathbb{R}\setminus\{1\}$ satisfying that $|c-1|$ is sufficiently small. Let $s_1,\dots,s_r\in\mathbb{C}$ with $\Re(s_k)>1$ and let $a_j\in\mathbb{R}$ with $-\Re(s_j)<a_j<-1$ $(t+1\leq j\leq r)$. Then, the integral
	\begin{align}\label{eqn:4.10}
		&\int_{(a_{t+1})\times\cdots\times(a_r)}\left\{\prod_{k=t+1}^r\frac{\Gamma(s_k+z_k)\Gamma(-z_k)}{\Gamma(s_k)}\right\} \frac{1}{(1-c)^r}\sum_{\delta_1,\dots,\delta_r\in\{0,1\}} \\
		&\scalebox{.95}{$\displaystyle \hspace{0.5cm}\cdot (-c)^{\delta_1+\cdots+\delta_r}\zeta_t\left(-z_1,\dots,-z_t;c^{\delta_1},\dots,c^{\delta_t}\right) \cdot\zeta_{r-t}(-z_{t+1},\dots,-z_{r};c^{\delta_{t+1}},\dots,c^{\delta_r})\prod_{l=t+1}^rdz_l.$}\nonumber
	\end{align}
	uniformly converges.
\end{lmm}
\begin{proof}
	We have
	\begin{align*}
		&\prod_{j=1}^t\left\{\frac{1}{\exp\left(\sum_{k=j}^tu_k\right)-1}
		-\frac{c}{\exp\left(c\sum_{k=j}^tu_k\right)-1}\right\} \\
		&\hspace{2cm}\cdot\prod_{j=t+1}^r\left\{\frac{1}{\exp\left(\sum_{k=j}^ru_k\right)-1}
		-\frac{c}{\exp\left(c\sum_{k=j}^ru_k\right)-1}\right\} \\
		&=\prod_{j=1}^t\left\{\sum_{\delta_j\in\{0,1\}}\frac{(-c)^{\delta_j}}{\exp\left(c^{\delta_j}\sum_{k=j}^tu_k\right)-1}\right\} 
		\cdot\prod_{j=t+1}^r\left\{\sum_{\delta_j\in\{0,1\}}\frac{(-c)^{\delta_j}}{\exp\left(c^{\delta_j}\sum_{k=j}^ru_k\right)-1}\right\} \\
		&=\sum_{\delta_1,\dots,\delta_r\in\{0,1\}} (-c)^{\delta_1+\cdots+\delta_r}
		\left\{\prod_{j=1}^t\frac{1}{\exp\left(c^{\delta_j}\sum_{k=j}^tu_k\right)-1}\right\}
		\left\{\prod_{j=t+1}^r\frac{1}{\exp\left(c^{\delta_j}\sum_{k=j}^ru_k\right)-1}\right\}.
	\end{align*}
	By using this and the following integral expression of $\zeta_r((s_j);(\gamma_j))$
	\begin{equation*}
		\zeta_r((s_j);(\gamma_j))=\prod_{k=1}^r\frac{1}{\Gamma(s_k)}\int_{(0,\infty)^r}\prod_{j=1}^r\frac{1}{\exp\left(\gamma_j\sum_{k=j}^ru_k\right)-1}\prod_{l=1}^ru_l^{s_l-1}du_l,
	\end{equation*}
	we have
	\begin{align*}
		&\frac{1}{(1-c)^r}\sum_{\delta_1,\dots,\delta_r\in\{0,1\}} (-c)^{\delta_1+\cdots+\delta_r}\zeta_t\left(-z_1,\dots,-z_t;c^{\delta_1},\dots,c^{\delta_t}\right) \\
		&\hspace{7.2cm}\vphantom{\frac{1}{(1-c)^r}\sum_{\delta_1,\dots,\delta_r\in\{0,1\}}}\cdot\zeta_{r-t}(-z_{t+1},\dots,-z_{r};c^{\delta_{t+1}},\dots,c^{\delta_r}) \\
		&=\frac{1}{(1-c)^r}\prod_{k=1}^r\frac{1}{\Gamma(-z_k)}\int_{(0,\infty)^r}
		\prod_{j=1}^t\left\{\frac{1}{\exp\left(\sum_{k=j}^tu_k\right)-1}
		-\frac{c}{\exp\left(c\sum_{k=j}^tu_k\right)-1}\right\} \\
		&\hspace{2.5cm}\cdot\prod_{j=t+1}^r\left\{\frac{1}{\exp\left(\sum_{k=j}^ru_k\right)-1}
		-\frac{c}{\exp\left(c\sum_{k=j}^ru_k\right)-1}\right\}
		\prod_{l=1}^ru_l^{-z_l-1}du_l.
	\end{align*}
	By \cite[Lemma 3.6]{FKMT1}, 
	for $c\in\mathbb{R}\setminus\{1\}$ such that $|c-1|$ is sufficiently small, we have a constant $A>0$ independent of $c$ such that
		\begin{equation*}
			\left|\frac{1}{c-1}\right|\left|\frac{1}{e^y-1}-\frac{c}{e^{cy}-1}\right|<Ae^{-y/2}
		\end{equation*}
		holds for any $y>0$. 
	Therefore, we get
	\begin{align*}
		&\left|\frac{1}{(1-c)^r}\sum_{\delta_1,\dots,\delta_r\in\{0,1\}} (-c)^{\delta_1+\cdots+\delta_r}\zeta_t\left(-z_1,\dots,-z_t;c^{\delta_1},\dots,c^{\delta_t}\right)\right. \\
		&\hspace{7.2cm}\left.\vphantom{\frac{1}{(1-c)^r}\sum_{\delta_1,\dots,\delta_r\in\{0,1\}}}\cdot\zeta_{r-t}(-z_{t+1},\dots,-z_{r};c^{\delta_{t+1}},\dots,c^{\delta_r})\right| \\
		&\leq\prod_{k=1}^r\frac{1}{\left|\Gamma(-z_k)\right|}\int_{(0,\infty)^r}
		\left\{\prod_{j=1}^tA\exp{\left(-\frac{1}{2}\sum_{k=j}^tu_k\right)}\right\} \\
		&\hspace{5cm}\cdot\left\{\prod_{j=t+1}^rA\exp{\left(-\frac{1}{2}\sum_{k=j}^ru_k\right)}\right\}
		\prod_{l=1}^ru_l^{-\Re(z_l)-1}du_l \\
		&=\prod_{k=1}^r\frac{A}{\left|\Gamma(-z_k)\right|}\int_{(0,\infty)^r}
		\exp{\left(-\frac{1}{2}\sum_{j=1}^t\sum_{k=j}^tu_k\right)}
		\exp{\left(-\frac{1}{2}\sum_{j=t+1}^r\sum_{k=j}^ru_k\right)}
		\prod_{l=1}^ru_l^{-\Re(z_l)-1}du_l \\
		&=\prod_{k=1}^r\frac{A}{\left|\Gamma(-z_k)\right|}\int_{(0,\infty)^r}
		\exp{\left(-\frac{1}{2}\sum_{k=1}^tku_k
		-\frac{1}{2}\sum_{k=t+1}^r(k-t)u_k\right)}
		\prod_{l=1}^ru_l^{-\Re(z_l)-1}du_l \\
		&=\left\{\prod_{k=1}^r\frac{A}{\left|\Gamma(-z_k)\right|}\right\}\prod_{k=1}^t\left\{\int_0^{\infty}
		\exp{\left(-\frac{k}{2}u_k\right)}
		u_k^{-\Re(z_k)-1}du_k\right\} \\
		&\hspace{5cm}\cdot\prod_{k=t+1}^r\left\{\int_0^{\infty}
		\exp{\left(-\frac{k-t}{2}u_k\right)}
		u_k^{-\Re(z_k)-1}du_k\right\}.
	\end{align*}
	Because we have
	\begin{equation*}
		n^{-s}\Gamma(s)=\int_0^{\infty}\exp(-nu)u^{s-1}du
	\end{equation*}
	 for $\Re(s)>0$ and $n\in\mathbb R_{>0}$ and we get $\Re(z_k)>0$ for $1\leq k\leq r$, we obtain the following inequality on the formula (\ref{eqn:4.10}):
	\begin{align*}
		&\left|\int_{(a_{t+1})\times\cdots\times(a_r)}\left\{\prod_{k=t+1}^r\frac{\Gamma(s_k+z_k)\Gamma(-z_k)}{\Gamma(s_k)}\right\} \frac{1}{(1-c)^r}\sum_{\delta_1,\dots,\delta_r\in\{0,1\}}\right. \\
		&\scalebox{.95}{$\left.\displaystyle \hspace{0.5cm}\cdot (-c)^{\delta_1+\cdots+\delta_r}\zeta_t\left(-z_1,\dots,-z_t;c^{\delta_1},\dots,c^{\delta_t}\right) \cdot\zeta_{r-t}(-z_{t+1},\dots,-z_{r};c^{\delta_{t+1}},\dots,c^{\delta_r})\prod_{l=t+1}^rdz_l\right|$} \\
		&\leq \int_{(a_{t+1})\times\cdots\times(a_r)}
		\left\{\prod_{k=t+1}^r\frac{|\Gamma(s_k+z_k)\Gamma(-z_k)|}{|\Gamma(s_k)|}\right\}
		\left\{\prod_{k=1}^r\frac{A}{\left|\Gamma(-z_k)\right|}\right\}\\
		&\hspace{2cm}\cdot
		\prod_{k=1}^t\left\{\left(\frac{k}{2}\right)^{\Re(z_k)}\Gamma\big(\Re(z_k)\big)\right\}
		\prod_{k=t+1}^r\left\{\left(\frac{k-t}{2}\right)^{\Re(z_k)}\Gamma\big(\Re(z_k)\big)\right\}\prod_{l=t+1}^r|dz_l|.
	\end{align*}
	On the above integral paths, we have $\Re(z_k)=a_k$ ($t+1\leq k\leq r$) and $-z_k=s_k$ ($1\leq k\leq t-1$) and $-z_t=s_t+\sum_{j=t+1}^r(s_j+z_j)$. So we put
	\begin{align*}
		C:=&A^r\left\{\prod_{k=t+1}^r\frac{1}{|\Gamma(s_k)|}\right\}
		\left\{\prod_{k=1}^{t-1}\frac{1}{\left|\Gamma(-z_k)\right|}\right\}\\
		&\hspace{1cm}\cdot\prod_{k=1}^t\left\{\left(\frac{k}{2}\right)^{\Re(z_k)}\Gamma\big(\Re(z_k)\big)\right\}
		\prod_{k=t+1}^r\left\{\left(\frac{k-t}{2}\right)^{\Re(z_k)}\Gamma\big(\Re(z_k)\big)\right\}.
	\end{align*}
	Then this symbol $C$ is independent on $z_{t+1},\dots,z_r$. Therefore, we get
	\begin{align*}
		&\left|\int_{(a_{t+1})\times\cdots\times(a_r)}\left\{\prod_{k=t+1}^r\frac{\Gamma(s_k+z_k)\Gamma(-z_k)}{\Gamma(s_k)}\right\} \frac{1}{(1-c)^r}\sum_{\delta_1,\dots,\delta_r\in\{0,1\}}\right. \\
		&\scalebox{.95}{$\left.\displaystyle \hspace{0.5cm}\cdot (-c)^{\delta_1+\cdots+\delta_r}\zeta_t\left(-z_1,\dots,-z_t;c^{\delta_1},\dots,c^{\delta_t}\right) \cdot\zeta_{r-t}(-z_{t+1},\dots,-z_{r};c^{\delta_{t+1}},\dots,c^{\delta_r})\prod_{l=t+1}^rdz_l\right|$} \nonumber\\
		&\leq C\int_{(a_{t+1})\times\cdots\times(a_r)}
		\prod_{k=t+1}^r|\Gamma(s_k+z_k)|
		\frac{1}{\left|\Gamma\left(s_t+\sum_{j=t+1}^r(s_j+z_j)\right)\right|}\prod_{l=t+1}^r|dz_l|.
	\end{align*}
	We have
	\begin{equation*}
		|\Gamma(\sigma+i\tau)|=\sqrt{2\pi}|\tau|^{\sigma-\frac{1}{2}}e^{-\frac{\pi}{2}|\tau|}\big(1+O(|\tau|^{-1})\big) \qquad (|\tau|\rightarrow\infty),
	\end{equation*}
	for $|\tau|\geq1$, where $O$ is the Landau symbol.
	So by using this equation, we get
	\begin{equation*}
		\int_{(a_{t+1})\times\cdots\times(a_r)}
		\prod_{k=t+1}^r|\Gamma(s_k+z_k)|
		\frac{1}{\left|\Gamma\left(s_t+\sum_{j=t+1}^r(s_j+z_j)\right)\right|}\prod_{l=t+1}^r|dz_l|<\infty.
	\end{equation*}
	We obtain the claim.
\end{proof}
The equation (\ref{eqn:4.3}) holds not only for $\zeta_r\left((s_j);(\gamma_j)\right)$ but also for $\zeta_r^{\rm des}(s_1,\dots,s_r)$.
\begin{prp}\label{prp:4.1}
	Let $s_1,\dots,s_r\in\mathbb{C}$ with $\Re{(s_j)}>1\ (1\leq j\leq r)$. Then, for  $-\Re{(s_k)}<a_k<-1\ (t+1\leq k\leq r)$ and $1\leq t\leq r-1$, we have
	\begin{align*}
		\zeta_r^{\rm des}(s_1,\dots,s_r) 
		=& \left(\frac{1}{2\pi i}\right)^{r-t}\int_{(a_{t+1})\times\cdots\times(a_r)}\prod_{k=t+1}^r\frac{\Gamma(s_k+z_k)\Gamma(-z_k)}{\Gamma(s_k)} \\
		&\cdot \zeta_t^{\rm des}\left(s_1,\dots,s_{t-1},s_t+\sum_{j=t+1}^r(s_j+z_j)\right)\zeta_{r-t}^{\rm des}(-z_{t+1},\dots,-z_{r})\prod_{l=t+1}^rdz_l.
	\end{align*}
\end{prp}
\begin{proof}
	We set $(-z_1,\dots,-z_t):=\left(s_1,\dots,s_{t-1},s_t+\sum_{j=t+1}^r(s_j+z_j)\right)$ and $\sigma_k:=-\Re(z_k)$ for $1\leq k\leq r$. By using Lemma \ref{lmm:4.1} and the above equation (\ref{eqn:4.3}), we get 
	{\small
	\begin{align}\label{eqn:4.7}
		&\zeta_r^{\rm des}(s_1,\dots,s_r) \\
		=&\lim_{\substack{c\rightarrow1 \\ c\in\mathbb{R}\setminus\{1\}}}\frac{1}{(1-c)^r}\sum_{\delta_1,\dots,\delta_r\in\{0,1\}}(-c)^{\delta_1+\cdots+\delta_r}\zeta_r((s_j);(c^{\delta_j})) \nonumber\\
		=& \lim_{\substack{c\rightarrow1 \\ c\in\mathbb{R}\setminus\{1\}}}\frac{1}{(1-c)^r}\sum_{\delta_1,\dots,\delta_r\in\{0,1\}}(-c)^{\delta_1+\cdots+\delta_r}\left(\frac{1}{2\pi i}\right)^{r-t}\int_{(a_{t+1})\times\cdots\times(a_r)}\prod_{k=t+1}^r\frac{\Gamma(s_k+z_k)\Gamma(-z_k)}{\Gamma(s_k)} \nonumber\\
		&\hspace{0.cm}\cdot \zeta_t\left(-z_1,\dots,-z_t;c^{\delta_1},\dots,c^{\delta_t}\right) 
		\cdot\zeta_{r-t}(-z_{t+1},\dots,-z_{r};c^{\delta_{t+1}},\dots,c^{\delta_r})\prod_{l=t+1}^rdz_l. \nonumber \\
		=& \lim_{\substack{c\rightarrow1 \\ c\in\mathbb{R}\setminus\{1\}}}\left(\frac{1}{2\pi i}\right)^{r-t}\int_{(a_{t+1})\times\cdots\times(a_r)}\prod_{k=t+1}^r\frac{\Gamma(s_k+z_k)\Gamma(-z_k)}{\Gamma(s_k)} \nonumber\\
		&\scalebox{.9}{$\displaystyle \hspace{0.cm}\cdot \frac{1}{(1-c)^r}\sum_{\delta_1,\dots,\delta_r\in\{0,1\}}(-c)^{\delta_1+\cdots+\delta_r}\zeta_t\left(-z_1,\dots,-z_t;c^{\delta_1},\dots,c^{\delta_t}\right) 
		\cdot\zeta_{r-t}(-z_{t+1},\dots,-z_{r};c^{\delta_{t+1}},\dots,c^{\delta_r})\prod_{l=t+1}^rdz_l.$} \nonumber
	\end{align}}
	By Lemma \ref{lmm:4.3} and Lebesgue's dominated convergence theorem, we can commute the limit $\displaystyle\lim_{\substack{c\rightarrow1 \\ c\in\mathbb{R}\setminus\{1\}}}$ with the integral $\displaystyle\int_{(a_{t+1})\times\cdots\times(a_r)}$.
	Therefore we have
	{\small
	\begin{align*}
	&\zeta_r^{\rm des}(s_1,\dots,s_r) \\
	=& \left(\frac{1}{2\pi i}\right)^{r-t}\int_{(a_{t+1})\times\cdots\times(a_r)}\prod_{j=t+1}^r\frac{\Gamma(s_j+z_j)\Gamma(-z_j)}{\Gamma(s_j)} \\
	&\cdot \left\{\lim_{\substack{c\rightarrow1 \\ c\in\mathbb{R}\setminus\{1\}}}\frac{1}{(1-c)^t}\sum_{\delta_1,\dots,\delta_t\in\{0,1\}}(-c)^{\delta_1+\cdots+\delta_t}\zeta_t\left(-z_1,\dots,-z_t;c^{\delta_1},\dots,c^{\delta_t}\right)\right\} \\
	&\cdot\left\{\lim_{\substack{c\rightarrow1 \\ c\in\mathbb{R}\setminus\{1\}}}\frac{1}{(1-c)^{r-t}}\sum_{\delta_{t+1},\dots,\delta_r\in\{0,1\}}(-c)^{\delta_{t+1}+\cdots+\delta_r}\zeta_{r-t}(-z_{t+1},\dots,-z_{r};c^{\delta_{t+1}},\dots,c^{\delta_r})\right\}  \prod_{l=t+1}^rdz_l\\
	=& \left(\frac{1}{2\pi i}\right)^{r-t}\int_{(a_{t+1})\times\cdots\times(a_r)}\prod_{k=t+1}^r\frac{\Gamma(s_k+z_k)\Gamma(-z_k)}{\Gamma(s_k)}  \zeta_t^{\rm des}(-z_1,\dots,-z_t)\zeta_{r-t}^{\rm des}(-z_{t+1},\dots,-z_{r})\prod_{l=t+1}^rdz_l.
\end{align*}}
So we obtain the claim.
\end{proof}

\begin{prp}\label{thm:4.1}
	Let $1\leq t\leq r$. For $s_1,\dots,s_t\in\mathbb{C}$ and $k_{t+1},\dots,k_r\in\mathbb{N}_0$, we have
	\begin{align*}
		&\zeta_r^{\rm des}(s_1,\dots,s_t,-k_{t+1},\dots,-k_r) \\
		&= \sum_{\substack{
			i_b + j_b=k_b \\
			i_b,j_b\geq0 \\
			t+1\leq b\leq r}}
		\prod_{a=t+1}^r\binom{k_a}{i_a} \zeta_{t}^{\rm des}(s_1,\dots,s_{t-1},s_t-i_{t+1}- \cdots -i_r)\zeta_{r-t}^{\rm des}(-j_{t+1},\dots,-j_r).
	\end{align*}
\end{prp}
\begin{proof}
	Let $s_1,\dots,s_r\in\mathbb{C}$ with $\Re(s_j)>1$ ($1\leq j\leq r$), $1\leq t\leq r-1$ and $a_{t+1},\dots,a_r\in\mathbb{R}$ with $-\Re(s_k)<a_k<-1$ ($t+1\leq k\leq r$). We assume $1\leq t\leq r-1$. To save space, we put
	{\small
	\begin{align*}
		&f(s_1,\dots,s_r;z_{t+1},\dots,z_r) 
		:=\zeta_t^{\rm des}\left(s_1,\dots,s_{t-1},s_t+\sum_{j=t+1}^r(s_j+z_j)\right)\zeta_{r-t}^{\rm des}(-z_{t+1},\dots,-z_{r}).
	\end{align*}}
	By using Proposition \ref{prp:4.1}, we have
	\begin{align*}
		\zeta_r^{\rm des}(s_1,\dots,s_r) 
		=& \left(\frac{1}{2\pi i}\right)^{r-t-1}\int_{(a_{t+1})\times\cdots\times(a_{r-1})}\prod_{j=t+1}^{r-1}\frac{\Gamma(s_j+z_j)\Gamma(-z_j)}{\Gamma(s_j)} \\
		&\cdot \left\{\frac{1}{2\pi i}\int_{(a_r)}\frac{\Gamma(s_r+z_r)\Gamma(-z_r)}{\Gamma(s_r)}f(s_1,\dots,s_r;z_{t+1},\dots,z_r)dz_r\right\}\prod_{l=t+1}^{r-1}dz_l.
	\end{align*}
	For $M_r\in\mathbb{N}$ and sufficiently small $\varepsilon_r>0$, we set $\mathcal{D}_r:=\{z_r\in\mathbb{C}\ |\ a_r<\Re(z_r)<M_r-\varepsilon_r\}$. For $z_r\in\mathcal{D}_r$, we have $\Re(s_r+z_r)>0$ by $-\Re(s_r)<a_r<0$. So singularities of the above integrand, which lie on $\mathcal{D}_r$, are only $z_r=0,1,\dots,M_r-1$. By using the residue theorem, we get
	\begin{align*}
		&\zeta_r^{\rm des}(s_1,\dots,s_r) \\
		&= \left(\frac{1}{2\pi i}\right)^{r-t-1}\int_{(a_{t+1})\times\cdots\times(a_{r-1})}\prod_{j=t+1}^{r-1}\frac{\Gamma(s_j+z_j)\Gamma(-z_j)}{\Gamma(s_j)} \\
		&\hspace{1cm}\cdot \left\{-\sum_{j_r=0}^{M_r-1}{\rm Res}\left[\frac{\Gamma(s_r+z_r)\Gamma(-z_r)}{\Gamma(s_r)}f(s_1,\dots,s_r;z_{t+1},\dots,z_r),z_r=j_r\right]\right. \\
		&\hspace{1cm}\left.+\frac{1}{2\pi i}\int_{(M_r-\varepsilon_r)}\frac{\Gamma(s_r+z_r)\Gamma(-z_r)}{\Gamma(s_r)}f(s_1,\dots,s_r;z_{t+1},\dots,z_r)dz_r\right\}\prod_{l=t+1}^{r-1}dz_l.
	\end{align*}
	(By the same arguments to those of \cite{Matsumoto}, the above second term converge). By using the fact that the residue of gamma function $\Gamma(s)$ at $s=-j$ is $\frac{(-1)^j}{j!}$, we have
	\begin{equation*}
		{\rm Res}\left[\frac{\Gamma(s_r+z_r)\Gamma(-z_r)}{\Gamma(s_r)},z_r=j_r\right]=(s_r+j_r-1)\cdots s_r\cdot\frac{(-1)^{j_r}}{j_r!}=\binom{-s_r}{j_r}.
	\end{equation*}
	So we obtain
	\begin{align*}
		&\zeta_r^{\rm des}(s_1,\dots,s_r) \\
		&= \left(\frac{1}{2\pi i}\right)^{r-t-1}\int_{(a_{t+1})\times\cdots\times(a_{r-1})}\prod_{j=t+1}^{r-1}\frac{\Gamma(s_j+z_j)\Gamma(-z_j)}{\Gamma(s_j)} \\
		&\hspace{1cm}\cdot \left\{\sum_{j_r=0}^{M_r-1}\binom{-s_r}{j_r}f(s_1,\dots,s_r;z_{t+1},\dots,z_{r-1},j_r)\right. \\
		&\hspace{1cm}\left.+\frac{1}{2\pi i\Gamma(s_r)}\int_{(M_r-\varepsilon_r)}\Gamma(s_r+z_r)\Gamma(-z_r)f(s_1,\dots,s_r;z_{t+1},\dots,z_r)dz_r\right\}\prod_{l=t+1}^{r-1}dz_l.
	\end{align*}
	By setting $s_r=-k_r$ and $M_r=k_r+1$ for $k_r\in\mathbb{N}_0$, because of $\frac{1}{\Gamma(-k_r)}=0$, we get
	{\small
	\begin{align*}
		\zeta_r^{\rm des}(s_1,\dots,s_{r-1},-k_r) 
		=& \left(\frac{1}{2\pi i}\right)^{r-t-1}\int_{(a_{t+1})\times\cdots\times(a_{r-1})}\prod_{j=t+1}^{r-1}\frac{\Gamma(s_j+z_j)\Gamma(-z_j)}{\Gamma(s_j)} \\
		&\cdot \left\{\sum_{j_r=0}^{k_r}\binom{k_r}{j_r}f(s_1,\dots,s_{r-1},-k_r;z_{t+1},\dots,z_{r-1},j_r)\right\}\prod_{l=t+1}^{r-1}dz_l.
	\end{align*}}
	In the same way, we have
	\begin{align*}
		&\zeta_r^{\rm des}(s_1,\dots,s_{r-2},-k_{r-1},-k_r) \\
		&= \left(\frac{1}{2\pi i}\right)^{r-t-2}\int_{(a_{t+1})\times\cdots\times(a_{r-2})}\prod_{j=t+1}^{r-2}\frac{\Gamma(s_j+z_j)\Gamma(-z_j)}{\Gamma(s_j)} \\
		&\scalebox{0.9}{$\displaystyle\cdot \left\{\sum_{j_r=0}^{k_r}\sum_{j_{r-1}=0}^{k_{r-1}}\binom{k_r}{j_r}\binom{k_{r-1}}{j_{r-1}}f(s_1,\dots,s_{r-2},-k_{r-1},-k_r;z_{t+1},\dots,z_{r-2},j_{r-1},j_r)\right\}\prod_{l=t+1}^{r-2}dz_l.$}
	\end{align*}
	By repeating the above computation, we get
	\begin{align*}
		&\zeta_r^{\rm des}(s_1,\dots,s_{t+1},-k_{t+2},\dots,-k_r) \\
		&= \frac{1}{2\pi i}\int_{(a_{t+1})}\frac{\Gamma(s_{t+1}+z_{t+1})\Gamma(-z_{t+1})}{\Gamma(s_{t+1})} \\
		&\scalebox{0.9}{$\displaystyle\cdot \left\{\sum_{j_r=0}^{k_r}\cdots\sum_{j_{t+2}=0}^{k_{t+2}}\binom{k_r}{j_r}\cdots\binom{k_{t+2}}{j_{t+2}}f(s_1,\dots,s_{t+1},-k_{t+2},\dots,-k_r;z_{t+1},j_{t+2},\dots,j_r)\right\}dz_{t+1}.$}
	\end{align*}
	By carrying out the above computation again, lastly we obtain
	\begin{align*}
		&\zeta_r^{\rm des}(s_1,\dots,s_t,-k_{t+1},\dots,-k_r) \\
		&=\sum_{j_r=0}^{k_r}\cdots\sum_{j_{t+1}=0}^{k_{t+1}}\binom{k_r}{j_r}\cdots\binom{k_{t+1}}{j_{t+1}}f(s_1,\dots,s_t,-k_{t+1},\dots,-k_r;j_{t+1},\dots,j_r).
	\end{align*}
	Therefore, we get the proposition for $(s_1,\dots,s_r)\in\mathbb{C}^r$ with $\Re(s_j)>1$.
	Because the function $\zeta_r^{\rm des}(s_1,\dots,s_r)$ is analytic on $\mathbb{C}^r$ we get the claim for $(s_1,\dots,s_r)\in\mathbb{C}^r$.
\end{proof}
\begin{lmm}\label{lem:4.2}
	Let $f,g:\mathbb{C}\times\mathbb{Z}^q\rightarrow\mathbb{C}$ be maps $(q\in\mathbb{N})$. We assume that
	\begin{equation}\label{eqn:4.4}
		g(s;-l_1,\dots,-l_q)=\sum_{\substack{
			i_b + j_b=l_b \\
			i_b,j_b\geq0 \\
			1\leq b \leq q}}
		\left\{\prod_{a=1}^q\binom{l_a}{i_a}\right\}
		\cdot f(s-i_1-\cdots-i_q;-j_1,\dots,-j_q)
	\end{equation}
	for $s\in\mathbb{C}$ and $l_1,\dots,l_q\in\mathbb{N}_0$. Then we have
	\begin{equation}\label{eqn:4.5}
		f(s;-l_1,\dots,-l_q)=\sum_{\substack{
			i_b + j_b=l_b \\
			i_b,j_b\geq0 \\
			1\leq b \leq q}}
		\left\{\prod_{a=1}^q(-1)^{i_a}\binom{l_a}{i_a}\right\}
		\cdot g(s-i_1-\cdots-i_q;-j_1,\dots,-j_q)
	\end{equation}
	for $s\in\mathbb{C}$ and $l_1,\dots,l_q\in\mathbb{N}_0$.
\end{lmm}
\begin{proof}
	Firstly, we prove this claim in the case of $q=1$ by induction on $l_1$. The case of $l_1=0$ is obvious. We assume the equation (\ref{eqn:4.5}) for $q=1$ and  $l_1\leq l-1$\quad ($l\in\mathbb{N}$). When $l_1=l$, from the equation (\ref{eqn:4.4}), we have
	\begin{align*}
		f(s;-l)=&g(s;-l)-\sum_{j=0}^{l-1}\binom{l}{j}f(s-l+j;-j).
		\intertext{By using the equation (\ref{eqn:4.5}), we get}
		=&g(s;-l)-\sum_{j=0}^{l-1}\binom{l_0}{j}\left\{\sum_{k=0}^j(-1)^k\binom{j}{k}g(s-l+j-k;-j+k)\right\} \\
		=&g(s;-l)-\sum_{j=0}^{l-1}\sum_{k=0}^j(-1)^k\binom{l}{j}\binom{j}{k}g(s-l+j-k;-j+k).
		\intertext{By putting $i=j-k$\quad($0\leq i\leq l-1$), we have}
		=&g(s;-l)-\sum_{i=0}^{l-1}\sum_{j=i}^{l-1}(-1)^{j-i}\binom{l}{j}\binom{j}{j-i}g(s-l+i;-i) \\
		=&g(s;-l)-\sum_{i=0}^{l-1}\left\{\sum_{j=i}^{l}(-1)^{j-i}\binom{l}{j}\binom{j}{i}-(-1)^{l-i}\binom{l}{i}\right\}g(s-l+i;-i) \\
		=&g(s;-l)+\sum_{i=0}^{l-1}(-1)^{l-i}\binom{l}{i}g(s-l+i;-i) \\
		=&\sum_{i=0}^{l}(-1)^{l-i}\binom{l}{i}g(s-l+i;-i).
	\end{align*}
	Secondly, we prove the claim for $q\geq2$. From the equation (\ref{eqn:4.4}), we have
	{\small
	\begin{align*}
		&g(s;-l_1,\dots,-l_q) \\
		&=\sum_{i_1 + j_1=l_1}\binom{l_1}{i_1}\left[\sum_{i_2 + j_2=l_2}\binom{l_2}{i_2}\left[\cdots\left[\sum_{i_q + j_q=l_q}\binom{l_q}{i_q}
		 f(s-i_1-\cdots-i_q;-j_1,\dots,-j_q)\right]\cdots\right]\right].
	\end{align*}}
	By using Lemma \ref{lem:4.2} as $q=1$, we get
	{\small
	\begin{align*}
		&\sum_{i_1 + j_1=l_1}(-1)^{i_1}\binom{l_1}{i_1}g(s-i_1;-j_1,-l_2\dots,-l_q) \\
		&=\sum_{i_2 + j_2=l_2}\binom{l_2}{i_2}\left[\cdots\left[\sum_{i_q + j_q=l_q}\binom{l_q}{i_q}
		 f(s-i_2-\cdots-i_q;-l_1,-j_2,\dots,-j_q)\right]\cdots\right].
	\end{align*}}
	By using Lemma \ref{lem:4.2} as $q=1$ again, we have
	{\small
	\begin{align*}
		&\sum_{i_2 + j_2=l_2}(-1)^{i_2}\binom{l_2}{i_2}\left[\sum_{i_1 + j_1=l_1}(-1)^{i_1}\binom{l_1}{i_1}g(s-i_1-i_2;-j_1,-j_2,-l_3\dots,-l_q)\right] \\
		&=\sum_{i_3 + j_3=l_3}\binom{l_3}{i_3}\left[\cdots\left[\sum_{i_q + j_q=l_q}\binom{l_q}{i_q}
		 f(s-i_3-\cdots-i_q;-l_1,-l_2-j_3,\dots,-j_q)\right]\cdots\right].
	\end{align*}}
	Therefore, by using Lemma \ref{lem:4.2} repeatedly, we obtain the claim.
\end{proof}
By Proposition \ref{thm:4.1} and Lemma \ref{lem:4.2}, we obtain the following theorem.
\begin{thm}\label{cor:4.1}
	For $s_1,\dots,s_p\in\mathbb{C}$ and $l_1,\dots,l_q\in\mathbb{N}_0$, we have
	\begin{align}\label{eqn:4.6}
		&\zeta_p^{\rm des}(s_1,\dots,s_p)\zeta_q^{\rm des}(-l_1,\dots,-l_q) \\
		&= \sum_{\substack{
			i_b + j_b=l_b \\
			i_b,j_b\geq0 \\
			1\leq b \leq q}}
		\prod_{a=1}^q(-1)^{i_a}\binom{l_a}{i_a} \zeta_{p+q}^{\rm des}(s_1,\dots,s_{p-1},s_p-i_1- \cdots -i_q,-j_1,\dots,-j_q). \nonumber
	\end{align}
\end{thm}
\begin{proof}
	By putting $r=p+q$, $t=p$ and $(k_{t+1},\dots,k_r):=(l_1,\dots,l_q)$ in Proposition \ref{thm:4.1}, we have
	\begin{align*}
		&\zeta_{p+q}^{\rm des}(s_1,\dots,s_p,-l_1,\dots,-l_q) \\
		&= \sum_{\substack{
			i_b + j_b=l_b \\
			i_b,j_b\geq0 \\
			1\leq b \leq q}}
		\prod_{a=1}^q\binom{l_a}{i_a} \zeta_{p}^{\rm des}(s_1,\dots,s_{p-1},s_p-i_1- \cdots -i_q)\zeta_{q}^{\rm des}(-j_1,\dots,-j_q).
	\end{align*}
	By applying Lemma \ref{lem:4.2} to the above equation with
	\begin{align*}
		g(s;-l_1,\dots,-l_q)&=\zeta_{p+q}^{\rm des}(s_1,\dots,s_{p-1},s,-l_1,\dots,-l_q), \\
		f(s;-l_1,\dots,-l_q)&=\zeta_p^{\rm des}(s_1,\dots,s_{p-1},s)\zeta_q^{\rm des}(-l_1,\dots,-l_q),
	\end{align*}
	we get the theorem.
\end{proof}

\thanks{ {\it Acknowledgements}. The author is cordially grateful to Professor H. Furusho for guiding him towards this topic and for giving  useful suggestions to him. He greatly appreciates K. Matsumoto, H. Tsumura and the referee of his previous paper \cite{Komi2} who gave him numerous and variable comments. This work was supported by JSPS JP18J14774 and the Research Institute for Mathematical Sciences, an International Joint Usage/Research Center located in Kyoto University.}


\end{document}